\documentclass[12pt,a4paper]{article}
\usepackage{amsthm}
\usepackage{color}
\usepackage{amsmath}

\usepackage{amssymb}
\usepackage{epsfig}
\usepackage{graphicx}
\usepackage[english]{babel}
\usepackage{plain}

\newtheorem{theorem}{Theorem}[section]
\newtheorem{lemma}[theorem]{Lemma}
\newtheorem{proposition}[theorem]{Proposition}
\newtheorem{corollary}[theorem]{Corollary}

\begin{document}
{
    \catcode`\@=11
    \gdef\trace{\mathop{\operator@font Tr}\nolimits}
    \gdef\cov{\mathop{\operator@font Cov}\nolimits}
    \gdef\var{\mathop{\operator@font Var}\nolimits}
    \gdef\E{\mathop{\operator@font \mathbb{E}}\nolimits}
    \gdef\L{\mathop{\operator@font \mathcal{L}}\nolimits}
}
\title{Continuous Data Assimilation with Stochastically Noisy Data}
\author{Hakima Bessaih,\footnote{University of Wyoming, Department of
    Mathematics, Dept. 3036, 1000 East University Avenue, Laramie WY
    82071, USA; {\it email:}{\tt\ bessaih@uwyo.edu}} \,\, Eric
  Olson\footnote{Department of Mathematics and Statistics, University
    of Nevada, Reno, NV 89557, USA. {\it email:}{\tt\
      ejolson@unr.edu}} \,\, and \,Edriss S.  Titi\footnote{Department
    of Computer Science and Applied Mathematics, Weizmann Institute of
    Science, Rehovot 76100, Israel; {\it email:}{\tt\
      edriss.titi@weizmann.ac.il}. ALSO: Department of Mathematics and
    Department of Mechanical and Aerospace Engineering, The University
    of California, Irvine, CA 92697, USA; {\it email:}{\tt\
      etiti@math.uci.edu}}}

\date{June 5th, 2014}
\maketitle
\centerline{\it Dedicated to Professor Ciprian Foias on the occasion of his $80^{th}$ birthday.}

\begin{abstract}
  We analyze the performance of a data-assimilation algorithm based on a
  linear feedback control when used with observational data that
  contains measurement errors.  Our model problem consists of dynamics
  governed by the two-dimension incompressible Navier--Stokes
  equations, observational measurements given by finite volume
  elements or nodal points of the velocity field and measurement
  errors which are represented by stochastic noise.  Under these
  assumptions, the data-assimilation algorithm consists of a system of
  stochastically forced Navier--Stokes equations.  The main result of
  this paper provides explicit conditions on the observation density (resolution) which guarantee explicit
  asymptotic bounds, as the time tends to infinity,  on the error between the approximate solution and the actual solutions which is corresponding to these measurements, in terms of  the variance of the noise in the measurements. Specifically, such bounds are given for the the limit supremum, as the time tends to infinity,
  of the expected value of the $L^2$-norm and of the $H^1$ Sobolev norm of the difference
  between the approximating solution and the actual solution.
  Moreover, results on the average time error in mean are stated.
  \end{abstract}

\section{Introduction}

Data assimilation is a process by which
a time series of observational data for a physical system
is used along with the
knowledge about the physics which govern the dynamics
to obtain an improved estimate of the current state of the system.
Applications of data assimilation
arise in many fields of geosciences, perhaps most importantly in
weather forecasting and hydrology.  The classical method of continuous
data assimilation, see, {\it e.g.}, Daley \cite{daley1991}, is to insert
observational measurements directly into a computer model as the
latter is being integrated in time.

A new approach, inspired by ideas
from control theory \cite{azouani2013}, has been proposed in
\cite{olson2013} that consists in introducing a feedback control term
that forces the approximating solution obtained by
data assimilation toward the
reference solution that is corresponding to the observations (see also \cite{FJKT1,FJKT1} for other applications).
This approach
admits a general framework of interpolant operators
which include operators arising from local volume averages and
pointwise nodal measurements.
Special attention is given to nodal measurements because they
may be likened to the data collected by an array of weather-vane
anemometers placed throughout the physical domain.
This is unlike previous rigorous work
\cite{browning1998,olson2011,henshaw2003,Korn,olson2003} where the
observed data is assumed to be the lower Fourier modes of the
reference unknown full solution.
The method of data assimilation introduced in \cite{olson2013}
extends equally to all
dissipative dynamical systems and relies on the fact that such
dynamical systems possess finite number of determining parameters,
such as determining modes, nodes and local volume averages,
see, for example,
\cite{cockburn1996,foias2001,FP,FT1,FT2,FTi,Holst-Titi,jones1992,jones1993} and
references therein.

With the exception of
Bl\"omker and coauthors~\cite{blomker2012}
for the 3DVAR Gaussian filter,
previous theoretical work
assumed that the observational measurements
are error free.
In this paper, we extend the approach of \cite{olson2013} to the case
where the observations are contaminated with random errors.
This allows us to treat measurement errors for
general interpolant observables and, in particular, when the
measurements are nodal values with random noise.
In this way our analysis may also be seen to extend the work
of \cite{blomker2012} from Fourier mode measurements to
general interpolant observables.

The method of data assimilation studied in this paper can be
described mathematically as follows.
Let $U(t)$ be a solution trajectory lying on the global attractor
of a known dissipative continuous dynamical system and
let $u(t)$ be the approximating solution
obtained from data assimilation of noisy observational measurements
of $U(t)$.
Assume that the dynamics of $U$ are governed by
an evolution equation of the form
\begin{plain}\begin{equation}
  {dU\over dt} = F(U)
  \label{E1}
\end{equation}\end{plain}%
with unknown initial condition
$U_{0}\in V$ at time $t_{0}$.  Here $F\colon V\to H$, where
$V$ and $H$ are suitable function spaces.

Denote by ${\cal O}_h(U(t))$, for $t\ge 0$,
the exact observational measurements without error of the true
solution $U$ at time $t$.
For two-dimensional physical domains we assume
${\cal O}_h\colon V\to \mathbb{R}^D$ to be linear operator
where $D$ is of the order $(L/h)^2$ and $L$ is a typical
large length scale of the physical domain of interest
and $h$ is the observation density or resolution.
Denote by $R_{h}(U(t))$ the interpolation
of the observational data, namely,
$$
	R_h(U(t))=\L_h\circ{\cal O}_h(U(t)),$$
where $\L_h\colon\mathbb{R}^D\to H$ is linear operator.
Note that $R_h$ need not be a projection
nor does its range
need be included in the domain of $F$.
Further assumptions on the general interpolant
observable $R_h$ are given in \eqref{R1} and \eqref{R2} below.

In the absence of measurement errors the data-assimilation
method proposed in \cite{olson2013} would construct
the approximating solution $u$ from the interpolant
observables $R_h(U(t))$ dynamically as the solution to
\begin{plain}\begin{equation}\label{noerror}
	{du\over dt}=F(u)-\mu(R_h(u)-R_h(U)),
\end{equation}\end{plain}%
with arbitrary initial condition $u(0)=u_0$.
Here $\mu>0$ is a relaxation parameter (nudging), whose value will
be determined later, which forces the coarse spatial scales
of $u$, {\it i.e.}, $R_h(u)$, toward those of the observed data,
{\it i.e.}, $R_h(U)$.

Suppose the exact measurements ${\cal O}_h(U(t))$
are subjected to some random errors.
Thus, the only observations available for
data assimilation
are the noisy observations $\tilde {\cal O}_h(U(t))$ given by
\begin{equation}\label{noisyO}
	\tilde {\cal O}_h(U(t)) = {\cal O}_h(U(t)) + {\cal E}(t),
\end{equation}
where ${\cal E}\colon [0,\infty)\to \mathbb{R}^D$ represents
the measurement error, for example, due to instrumental errors.
That is, in reality the actual interpolated
measurements of $U(t)$ contain random errors and are given
by
\begin{plain}$$\eqalign{
	\tilde R_h(U(t))
		&=\L_h(\tilde {\cal O}_h(U(t)))\cr
		&=\L_h({\cal O}_h(U(t))+\L_h({\cal E}(t))
		=R_h(U(t)) + \xi(t),
}$$\end{plain}%
where the random vector $\xi(t)$ lies in the range of the
interpolant operator $R_h$.

We will assume that the components of the random errors ${\cal E}(t)$
are independent identically distributed of Gaussian type.
In particular, the random error $\xi(t)$ will be
expressed in terms of a finite-dimensional Wiener process $W$,
white noise in time with an appropriate covariance operator. The
precise assumptions will be given in the following sections. We observe that these
results could be generalized to other kinds of
processes such as a Levy noise.

In this paper we examine the data-assimilation method given
by equation \eqref{noerror}
when the noise-free
interpolant observable $R_h(U(t))$ is replaced by $\tilde R_h(U(t))$.
In this case, our algorithm for
constructing $u(t)$ from the observational measurements
$\tilde{\cal O}_h(U(t))$ is given by the stochastic
evolution equation
\begin{equation}
  du = \left(F(u) -\mu R_h(u) +\mu R_h(U)\right)dt+\mu\xi dt,
  \label{approxeqn}
\end{equation}
with arbitrary initial condition $u(0)=u_0$.

The two-dimensional incompressible Navier--Stokes equations, subject
to period boundary conditions, provide a concrete example of a
dissipative dynamical system, which we will use as a
model problem for our analysis.
We find explicit conditions on the observation
density or resolution $h$ and relaxation parameter $\mu$
which guarantees that the
resulting approximate solution $u(t)$ converges, as
$t\to\infty$, in some sense, to the exact reference solution $U(t)$
within an error that is determined
by the observation density $h$, the relaxation parameter $\mu$ and
the variance of the error in the measurements.
It is worth mentioning that the application
of algorithm \eqref{approxeqn} to recover solutions to fluid
flow problems provides a concrete and justifiable reason
for investigating stochastically forced equations such as the
Navier-Stokes equations.

In the remainder of this work, the reference solution $U$ will be determined by
the two-dimensional incompressible Navier--Stokes equations
\begin{plain}\begin{equation}\label{NSE}
\left\{\eqalign{
  \frac{\partial U}{\partial t}- \nu \Delta U + (U \cdot\nabla)U
  =-\nabla p +f \cr
  \nabla \cdot U =0,
\cr
}\right.\end{equation}\end{plain}%
which describe the motion of an incompressible fluid in
$\mathbb{R}^2$.
We assume periodic boundary conditions with fundamental domain
$\mathcal{D}=[0,L]^{2}$ and take the initial condition
$U(0,x)=U_{0}(x)$ and the
body forcing $f=f(x)$ to be an $L$-periodic function with zero
spatial average.
The kinematic viscosity $\nu>0$ is assumed to be known.
The unknowns are the velocity vector $U=U(t,x)$
and scalar pressure $p=p(t,x)$.
We observe that \eqref{NSE} preserves the $L$-periodicity
and zero spatial average of the initial condition.
Thus $\int_\mathcal{D} U(t,x) \,dx=0$,
for all $t \ge 0$, provided $\int_\mathcal{D} f(x) \, dx=
\int_\mathcal{D} U_0(x) \,dx=0$, which we will assume throughout this
paper.

For notational convenience we will denote $L^1_{\rm per}$
as simply $L^1$, and similarly $L^2_{\rm per}$ by $L^2$
and $H^1_{\rm per}({\cal D})$ by $H^1$.
For $\varphi\in L^1$ we define the average
$$
\langle\varphi\rangle =
	\frac{1}{L^2}\int_{\mathcal{D}}\varphi(x)\, dx,
$$
and for every $\mathcal{Z} \subseteq
L^1$ we denote
$\dot{\mathcal{Z}} =\{\,
\varphi \in \mathcal{Z}: \langle\varphi\rangle=0\,\}$.

In the absence of measurement errors the data assimilation
method given by \eqref{noerror} for the two-dimensional
incompressible Navier-Stokes equations allows the use of two kinds of linear interpolant observables. The first kind are
$R_{h}\colon [\dot{H}^{1}]^2\to [\dot{L}^{2}]^2$,
which satisfy the approximating identity property
\begin{equation}\label{R1}
  \|\varphi-R_{h}(\varphi)\|_{L^{2}}^{2}\leq
		c_1h^{2}\|\varphi\|_{H^{1}}^{2},
\end{equation}
for every $\varphi\in [\dot{H}^1]^2$; and the second kind
of interpolant observables  are $R_{h}\colon [\dot{H}^{2}]^2\to
[\dot{L}^{2}]^2$, which satisfy
\begin{equation}\label{R2}
  \|\varphi-R_{h}(\varphi)\|_{L^{2}}^{2}\leq c_1 h^{2}\|\varphi\|_{H^{1}}^{2}+
   c_2 h^{4}\|\varphi\|_{H^{2}}^{2},
\end{equation}
for every $\varphi\in [\dot{H}^2]^2$.
In the presence of measurement errors this same data assimilation
method becomes the stochastic differential equation \eqref{noerror}
and our analysis needs additional regularity assumptions on $R_h$
for interpolants that satisfy \eqref{R2}.
In particular, we assume the range of $R_h$
is in $[\dot H^1]^2$ for interpolants which satisfy \eqref{R2}.
This does not
result in loss of generality because any interpolant
operator whose range is in $[\dot L^2]^2$ can be
smoothed so its range is in $[\dot H^{s}]^2$, for any
$s >0$.

The orthogonal
projection onto the Fourier modes, with wave numbers~$k$
such that $|k|\le 1/h$, is an example of an interpolant
operator which satisfies both approximation properties \eqref{R1} and \eqref{R2}.
A physically relevant interpolant which
satisfies (\ref{R1}) is given by the volume elements studied in
\cite{jones1992} and \cite{jones1993}, see also \cite{FTi}.
Suppose the observations of volume elements
${\cal O}_h\colon [\dot H^1]^2\to\mathbb{R}^{2N}$
are given by
\begin{plain}\begin{equation}\label{volumes}
{\cal O}_h(\Phi)=(\bar \varphi_1,\bar \varphi_2,\ldots,
		\bar\varphi_{2N})
\quad\hbox{where}\quad
\left[\matrix{\bar \varphi_{2n-1}\cr \bar \varphi_{2n}}\right] =
	{N\over L^2} \int_{Q_n} \Phi(x)\,dx,
\end{equation}\end{plain}%
for $n=1,2,\ldots,N$, where the
domain $\mathcal{D}=[0,L]^2$ has been divided into $N=K^2$
disjoint equal squares $Q_n$ with sides $h=L/K$.
Define $R_h=\L_h\circ {\cal O}_h$, where $\L_h:\mathbb{R}^{2N} \to {[\dot{L}^2(\cal D)]}^2$ with  $\L_h(\zeta)$ is the
$L$-periodic function given by
\begin{plain}\begin{equation}\label{stepL}
	\L_h(\zeta)(x)= \sum_{n=1}^{N} \left[\matrix{\zeta_{2n-1} \cr \zeta_{2n}}\right]
	\Big(\chi_{Q_n}(x)-\frac{h^{2}}{L^{2}}\Big)
\quad\hbox{on}\quad  {\cal D}.
\end{equation}\end{plain}%

As shown in \cite{jones1992} the interpolant $R_h$
satisfies \eqref{R1}, with $c_1=1/6$.
Note there are many other
choices for $\L_h$ that result in interpolant observables based on
volume elements which also satisfy (\ref{R1}).
For example, the appendix of \cite{olson2013}, which will be
summarized in section \ref{subsecerror} below,
presents a smoothed choice for $\L_h$ which results
in an $R_h\colon [\dot H^1]^2\to [\dot H^{2}]^2$ which
also satisfies \eqref{R1}.
In addition, volume elements
generalize to any domain $\mathcal{D}$ on which the Bramble--Hilbert
inequality holds.  An elementary discussion of this inequality in the
context of finite element methods appears in Brenner and Scott
\cite{brennerscott2007}, see also \cite{Ciarlet, Holst-Titi, WM}.

An interpolant observable
$R_h\colon [\dot H^2]^2\to [\dot H^{2}]^2$
which satisfies \eqref{R2} is obtained,
following the ideas of \cite{jones1993},
when the observational measurements are given by nodal measurements
of the velocity.
This corresponds to the data collected from an array of
weather-vane anemometers placed throughout the physical domain.
Suppose the observations of nodes
${\cal O}_h\colon [\dot H^2]^2\to\mathbb{R}^{2N}$ are given by
\begin{plain}\begin{equation}\label{nodes}
	{\cal O}_h(\Phi)= (\varphi_1,\varphi_2,\ldots,\varphi_{2N})
\quad\hbox{where}\quad
	\left[\matrix{\varphi_{2n-1}\cr\varphi_{2n}}\right]=\Phi(x_n),
\end{equation}\end{plain}%
and $x_n\in Q_n$, for $n=1,2,\ldots,N$.  Here
$Q_n$ are, as above, $N=K^2$ disjoint squares with sides
$h=L/K$ such that $\mathcal{D}=\cup_{n=1}^N Q_n$.
Setting $R_h=\L_h\circ {\cal O}_h$, where $\L_h$ is the smoothed
version of \eqref{stepL} given in section \ref{subsecerror},
results in an interpolant which satisfies \eqref{R2}.

The rest of this paper is organized as follows:
section \ref{secprelim}
describes the functional setting for our analysis,
gives the stochastic setting for our measurement errors
and recalls the {\it a-priori\/} estimates for classical solutions
of the two-dimensional incompressible
Navier--Stokes equations that we shall use to obtain our bounds.
Section \ref{secwp} shows the stochastic data assimilation
algorithm given by \eqref{approxeqn} is well-posed.
Section \ref{secmain} states and proves our main
results.
We end with a few concluding remarks.

\section{Preliminaries}\label{secprelim}

The preliminaries have been divided into three subsections.
Subsection \ref{subsecfunct} sets up notation and the functional
setting we will use in our analysis.
Subsection \ref{subsecerror} gives the stochastic
setting for our measurement errors and summarizes the
specific details from \cite{olson2013}
on the general interpolant observables needed for our
analysis.
Subsection \ref{subsecclass}
recalls the theory and {\it a-priori\/} estimates
for classical solutions of the two-dimensional incompressible
Navier--Stokes equations needed for our work.

\def\vl{(\!(}
\def\vr{)\!)}
\subsection{The Functional Setting}\label{subsecfunct}
We describe the functional setting which will be
used to study the Navier-Stokes equations.
We refer to \cite{constantin1988, robinson2001, temam1983, temam2001} for
the main results.
Denote by $\cal V$ all divergence-free $\mathbb{R}^2$
valued $L$-periodic trigonometric polynomials with zero
spatial averages.
Let $H$ and $V$ be the closures of $\mathcal{V}$ in
$[L^{2}]^{2}$ and $[H^1]^{2}$, respectively.
Note that $H$ and $V$ are separable Hilbert spaces with the inner products and
norms inherited from $[{L}^{2}]^{2}$ and
$[{H}^1]^{2}$, respectively.  In particular,
\[
|u|^2_H=\langle u,u\rangle,\, \quad \mbox{ \rm where} \quad \langle
u,v\rangle=\int_{\mathcal D} \big( u(x) \cdot v(x)\big) dx,
\]
and, thanks to the Poincar\'e inequality \eqref{poincare1},
\[
\|u\|_V^2 = \vl u,u\vr, \quad\mbox{\rm where}\quad
	\vl u,v\vr = \int_{\mathcal D}
\big( \nabla u(x):\nabla v(x)\big) dx.
\]

Denote by $H'$ and $V'$ the dual spaces of $H$ and $V$, respectively.
If we identify $H$ with
$H'$, then we have the Gelfand triple $V\subset H\subset V'$ with
continuous, compact and dense injections.  We denote the dual pairing
between $\varphi\in V'$ and $\psi\in V$ by $\langle
\varphi,\psi\rangle_{V',V}$.  When $\varphi \in H$, we have $\langle
\varphi,\psi\rangle_{V',V}=\langle \varphi,\psi\rangle$.

Let $\Pi$ be the Leray--Helmholtz projector from
$[\dot{L}^{2}]^{2}$ onto $H$. The Stokes operator
$A$ is given by
$$
Au=-\Pi \Delta u \quad \hbox{ for every}\quad u \in
	 D(A)=[\dot{H}^{2}]^{2}\cap V.
$$
Note that $A$ is a closed, positive, unbounded self-adjoint operator
in $H$ with inverse $A^{-1}$ which is a self-adjoint compact
operator on $H$.  By the spectral theorem there exists a
sequence $\{\lambda_j\}_{j=1}^{\infty}$ of eigenvalues of the Stokes
operator, with $0<\lambda_{1}\leq \lambda_{2}\leq \cdots$, with
corresponding eigenvectors $e_{j} \in D(A)$ such that
the set $\{e_j:j\in\mathbb{N}\}$
forms an orthonormal basis in $H$.  Moreover, we have
$\lambda_j \sim \lambda_1 j$, as $j\to \infty$, where
$\lambda_1=(2\pi/L)^2$ ({\it cf.} \cite{constantin1988}).

For $\alpha >0$ we will denote the $\alpha$-th power of the operator $A$
by $A^\alpha$ and its domain by $D(A^\alpha)$. We have
$\|u\|^2_{D(A^\alpha)}= \sum_{j=1}^\infty \lambda_j^{2\alpha}
|\langle u,e_j\rangle|^2$.  Moreover, it follows that $V=D(A^{1/2})$ with
$\|\varphi\|_{V}=|A^{1/2}\varphi|_{H}$, for every $\varphi\in V$, and
$D(A^{\alpha_1})$ is compactly embedded in $D(A^{\alpha_2})$, for
$\alpha_1>\alpha_2$. Finally, let $D(A^{-\alpha})$ denote the dual of
$D(A^{\alpha})$.

We have the following Poincar\'e inequalities:

\begin{equation}\label{poincare1}
  |u|_{H}^{2}\leq \lambda_{1}^{-1}\|u\|_V^{2}
	\quad\hbox{for}\quad u\in V
\end{equation}
and
\begin{equation}\label{poincare2}
  \|u\|_V^{2}\leq \lambda_{1}^{-1}|u|_{D(A)}^{2}
	\quad\hbox{for}\quad u\in D(A).
\end{equation}
Let $b(\cdot,\cdot,\cdot)\colon V\times V\times V\to
\mathbb{R}$ be the continuous trilinear form defined as
$$
b(u,v,z)=\int_{\mathcal D}\big((u(x)\cdot\nabla) v(x)\big)\cdot z(x)\, dx .
$$
It is well known that there exists a continuous bilinear operator
$B(\cdot,\cdot)\colon V\times V\to V'$ such that $\langle
B(u,v),z\rangle_{V',V} =b(u,v,z),\ {\rm for}\ {\rm all}\ z\in V.$

\begin{lemma}\label{Blem}(cf. \cite{constantin1988,robinson2001,temam1983,temam2001})
  Let $u,v, z\in V$. Then
  \begin{equation} \label{B} \langle B(u,v),z\rangle_{V',V}=- \langle
    B(u,z),v\rangle_{V',V} \quad \mbox{\rm and}\quad \langle
    B(u,v),v\rangle_{V',V} =0.
  \end{equation}
  Furthermore,
  \begin{equation}\label{bilinear_estimate1}
    |\langle B(u,v), z\rangle_{V',V}|
		\le \|u\|_{L^{4}}\|v\|_{V}\|z\|_{L^{4}}.
  \end{equation}
\end{lemma}
Moreover, one can apply the two-dimensional Ladyzhenskaya
interpolation inequality (cf. \cite{constantin1988})
  \begin{equation}\label{inter}
    \|u\|^{2}_{L^{4}}\leq C_{\rm L}|u|_{H}\|u\|_{V},
  \end{equation}
  to the right-hand side of \eqref{bilinear_estimate1} to obtain
\begin{equation}\label{bilinest}
|\langle B(u,v), z\rangle_{V',V}| \leq C_{\rm L}
|u|_{H}^{1/2}\|u\|_{V}^{1/2} \|v\|_{V}|z|_{H}^{1/2}\|z\|_{V}^{1/2},
\end{equation}
for functions in $V$.

We will also make use of the Br\'ezis--Gallouet inequality \cite{brezis1980},
which may be stated as
\begin{plain}\begin{equation}\label{brezis}
    \|v\|_{\infty}
    \leq C_{\rm B} \|v\|_V\bigg\{ 1+\log {|Av|_{H}^2\over \lambda_1
      \|v\|_{V}^2}\bigg\},
\end{equation}\end{plain}%
for functions in $D(A)$.

\begin{lemma}\label{B2lem}(cf. \cite{constantin1988,temam1997})
In the case of periodic boundary conditions the bilinear term has the
additional orthogonality property
\begin{equation}\label{B1}
  \langle B(v,v),Av\rangle =0,
\end{equation}
for every $v\in D(A)$. In addition, one has
\begin{equation}\label{B2}
  \langle B(u,v),Av\rangle + \langle B(v,u),Av\rangle= -\langle B(v,v),Au\rangle,
\end{equation}
for every $u,v\in D(A)$.
\end{lemma}

Applying the Leray-Helmholtz projector $\Pi$ to \eqref{NSE} one
obtains the equivalent functional evolution equation
\begin{equation}\label{abstract}
      \frac{dU}{dt}+\nu AU+B(U,U)= f,
\end{equation}
with initial condition $U({0})=U_{0}$, where
we assume that $f\in H$ and $U_0\in V$.  Similarly
the data-assimilation equation \eqref{approxeqn} becomes
\begin{plain}\begin{equation}\label{abapprox}
	\eqalign{
	du+(\nu Au &+B(u,u))dt
		= \big(f -\mu \Pi  R_h(u-U)\big)dt+\mu dW,
}\end{equation}\end{plain}%
where $dW(t)=\Pi\xi(t)dt$ is the noise term.

\subsection{The Noise Term}\label{subsecerror}

In this section we describe the error
term ${\cal E}\colon [0,\infty)\to R^{2N}$ that
gives rise to the noisy observations $\tilde {\cal O}_h$
in equation \eqref{noisyO} in terms of Brownian motions.
We then
use the definition
$\tilde R_h=\L_h\circ \tilde {\cal O}_h$
to obtain $dW$ in \eqref{abapprox}.

Following Da Prato and Zabczyk \cite{daprato1992}
fix a filtered probability space
$(\Omega, \mathcal{F}, (\mathcal{F}_t), \mathbb{P})$
on which is defined is a sequence of independent
one-dim\-en\-sional Brownian motions $b_d(t)$, for $d=1,2,\ldots,D$,
relative to the filtration $(\mathcal{F}_t)$ such
that
\begin{plain}$$
\E(b_d(t))=0
\quad\hbox{and}\quad
\E(b_d(t)^2) =  t\sigma^2/2
\quad\hbox{for}\quad t\ge 0.
$$\end{plain}%
For convenience we shall assume the filtration is complete
and right continuous.  The measurement errors may now be described by
\begin{equation}\label{noise}
	{\cal E}(t)dt=(db_1(t),db_2(t),\ldots,db_{D}(t)).
\end{equation}

Note that $\sigma$ is a dimensional constant whose
units of measurement must be chosen so the units of
measurement for ${\cal O}_h(U(t))$ are the same as ${\cal E}$.
Given a quantity $z$ let $[z]$ represent the units
used to measure $z$.  Then $[{\cal O}_h]=[{\cal E}]$ implies
$[\sigma^2]=[{\cal O}_h]^2[t]$.
In particular, if our observations are
velocities as in~\eqref{volumes} or~\eqref{nodes}, we then have
$[{\cal O}_h]=[L]/[t]$ so that
$[\sigma^2]=[L]^2/[t]$.

Writing the linear operator
$\L_h\colon \mathbb{R}^D\to [\dot H^{\alpha}]^2$, for $\alpha\ge 0$, as
\begin{equation}\label{genL}
	\L_h(\zeta)(\cdot)=\sum_{d=1}^{D} \zeta_d \ell_d(\cdot),
\quad\hbox{where}\quad \zeta \in \mathbb{R}^D \quad\hbox{and}\quad
\ell_d\in [\dot H^\alpha]^2,
\end{equation}
it follows that the noise term in \eqref{abapprox}
is the Wiener process
\begin{equation}\label{wiener}
	W(t)=\sum_{d=1}^D b_d(t)\gamma_d,
\quad\hbox{where}\quad
	\gamma_d=\Pi\ell_d.
\end{equation}
We do not assume $\gamma_d$ are orthogonal or
even linearly independent.

When $\alpha\ge 0$ our assumptions dictate that $W$ is a
$[\dot{H}^{\alpha}]^{2}$-valued
$Q$-Brownian motion, where $\E(W(t))=0$.
Following \cite{daprato1992} pages 26--27, we have
\begin{plain}$$\eqalign{
	tQ={\cov}(W(t))
	=\E\Big(\sum_{d,p=1}^D b_d(t)\gamma_d
			\otimes b_p(t)\gamma_p\Big).
}$$\end{plain}%
Note that
$Q$ is a nonnegative and symmetric linear operator with
finite trace.  In particular, we have
\begin{plain}$$
\eqalign{
  \trace\big[{\cov}(W(t))\big]
	&=\sum_{j=1}^{\infty}\big\langle
		{\rm Cov}(W(t))e_{j},e_{j}\big\rangle\nonumber\cr
    &=\sum_{j=1}^{\infty}\E\left(\sum_{d,p=1}^D
		\langle  b_d(t)\gamma_d, e_{j}\rangle
    \langle b_p(t)\gamma_p, e_{j}\rangle\right)\nonumber\cr
    &=\sum_{j=1}^{\infty}\left(\sum_{d,p=1}^D
		\E\left(b_d(t) b_p(t)\right)
    \langle \gamma_d, e_{j}\rangle \langle\gamma_p,
		e_{j}\rangle\right)\nonumber\cr
    &=t{\sigma^2\over 2}\sum_{j=1}^{\infty}\sum_{d=1}^D
    |\langle\gamma_d, e_{j}\rangle|^2\nonumber
    =t{\sigma^2\over 2}\sum_{d=1}^D |\gamma_d|_H^2.
}
$$\end{plain}%
Therefore,
\begin{plain}\begin{equation}\label{absTRQ}
  \trace[Q]
    ={\sigma^2\over 2}\sum_{d=1}^D |\gamma_d|_H^2<\infty.
\end{equation}\end{plain}%

We next turn our attention to the
specific interpolant observable based on volume
elements given by \eqref{volumes} and \eqref{stepL}.
In this case, setting
\begin{plain}\begin{equation}\label{stepell}
\ell_{2n-1}(x)=\left[\matrix{\chi_{Q_n}(x)-h^2/L^2\cr 0}\right]
\quad\hbox{and}\quad
\ell_{2n-1}(x)=\left[\matrix{0\cr \chi_{Q_n}(x)-h^2/L^2}\right],
\end{equation}\end{plain}%
yield, for $n=1,2,\ldots,N$, the $D=2N$ functions needed
in \eqref{genL} and we obtain	

\begin{proposition}\label{traceH}
Let $W(t)$ be defined as in \eqref{wiener}, where $\ell_d$
are given by \eqref{stepell},
for $d=1,2,\ldots,2N$.
Then $W$ is a $[\dot L^2]^2$-valued $Q$-Brownian motion with
covariance operator $Q$ that satisfies
$\trace[Q] \le \sigma^2 L^2.$
\end{proposition}
\begin{proof}
The calculation
\begin{plain}$$\eqalign{
  \trace[Q]
	&={\sigma^2\over 2}\sum_{d=1}^{2N} |\gamma_d|_H^2
		\le{\sigma^2\over 2}\sum_{d=1}^{2N} |\ell_d|_{L^2}^2
	=\sigma^2\sum_{n=1}^N
		\int_{\cal D}\Big|\chi_{Q_n}(x)-{h^2\over L^2}\Big|^2 dx\cr
	&=\sigma^2\sum_{n=1}^N
		\int_{\cal D}\Big\{\Big(1-{2h^2\over L^2}\Big)\chi_{Q_n}(x)
			+{h^4\over L^4}\Big\} dx
	\le \sigma^2(L^2-h^2)
	\le \sigma^2 L^2	
}$$\end{plain}%
immediately yields the result.
\end{proof}

We now recall the construction of the smooth interpolant
observables used for nodal measurements
that were constructed in the appendix of~\cite{olson2013}, and
which satisfy~\eqref{R2}. Then we derive the estimates needed in our analysis of~\eqref{abapprox}
for the terms resulting from the It\^o formula.

Let $Q_n$, for $n=1,2,\ldots, N$, be the $N=K^2$ squares with
sides $h=L/K$ described in the introduction
such that ${\cal D}=\cup_{i=1}^N Q_n$.
In particular, we set ${\cal J}= \{\, 1,\ldots,K\,\}^2$
and for $(i,j)\in {\cal J}$
define
\begin{equation}\label{squares}
Q_n=[(i-1)h,ih)\times [(j-1)h,jh),
\end{equation}
where $n=i+(j-1) K.$
Further define
\begin{equation}\label{charper}
\psi_n(x)=\sum_{k\in\mathbb{Z}^2}
\chi_{Q_n}(x+kL), \quad {\rm for}\quad x\in
\mathbb{R}^{2},
\end{equation}
as the $L$-periodicized characteristic function of $Q_n$.
Note that $\psi_n\in L^2$, and moreover, that
$ \langle \psi_n^2\rangle =\langle \psi_n\rangle = h^2/L^2.  $

To obtain a smoother interpolant let
$$\tilde \psi_n(x)=
(\rho_{h/10}*\psi_n)(x)$$ be the mollified version of
$\psi_n$, where
$\rho_\epsilon(x)=\epsilon^{-2}\rho(x/\epsilon)$, and
\begin{plain}\begin{eqnarray*}
  \rho(z)=\left\{\begin{array}{lr}
      K_0 \displaystyle
      \exp\Big({1\over 1-|z|^2}\Big)&   |z|<1\\
      0 & |z|\geq 1
    \end{array}
  \right.
\end{eqnarray*}\end{plain}%
with
\begin{plain}$$
	(K_0)^{-1}=\int_{|z|<1} \exp\Big({1\over 1-|z|^2}\Big)\,dz.
$$\end{plain}%
Now setting
\begin{plain}\begin{equation}\label{smoothell}
	\ell_{2n-1} = \left[\matrix{\tilde \psi_{n}
		-\langle\tilde\psi_{n}\rangle\cr 0}\right]
\quad\hbox{and}\quad
	\ell_{2n} = \left[\matrix{0\cr\tilde\psi_{n}
		-\langle\tilde\psi_{n}\rangle}\right],
\end{equation}\end{plain}%
for $n=1,2,\ldots,N$, yields the $D=2N$ functions needed
in \eqref{genL}
for the definition of $\L_h$.
As shown in the appendix of \cite{olson2013},
if the observations are given by volume elements,
then the resulting interpolant satisfies \eqref{R1};
if the observations are given by nodal points,
then the resulting interpolant satisfies~\eqref{R2}.

We finish this section with some explicit estimates on the
trace of the covariance operator $Q$, for the Wiener process
$W$ given in \eqref{wiener}
for the choice of $\ell_n$ given in \eqref{smoothell}.
Before that, we state two propositions which we
shall make use of in the proof as well as in other parts of this paper.
Detailed proofs of these propositions appear in the appendix
of \cite{olson2013}, where
the functions $\tilde \psi_n$ have been introduced along
with their associated interpolant observables.

\begin{proposition}\label{prop1}
Let
$$\mathcal{U}_n =\{\, x+y : x\in Q_n\hbox{ and\/ }|y|<\epsilon\,\},
\quad\hbox{for}\quad n=1,2,\ldots N.
$$
Then $\{\,\tilde\psi_n:n=1,2,\ldots,N\,\}$ is a
smooth partition of unity satisfying\par
  \begin{enumerate}
  \item[(i)] $0 \le \tilde\psi_n(x)\le 1$ and ${\rm
      supp}(\tilde\psi_n)\subseteq \mathcal{U}_n+ \big
    (L \mathbb{ Z}\big)^2$,
  \item[(ii)] $\tilde\psi_n(x)=1$, \, for all \, $x\in
    \big(\mathcal{C}_n+ \big(L \mathbb{Z}\big)^2\big)$ and \\ \hbox{} \\
	\hbox{\qquad}
    \,$\sum_{n=1}^N\tilde\psi_n(x)=1$, \, for all $x\in
    \mathbb{R}^2$,
  \item[(iii)] $\langle \tilde\psi_n \rangle=
    \big({h/L}\big)^2$ and \, $\frac{4}{5} h \le
    \big\|\tilde\psi_n \big\|_{L^2(D)} \le \frac{6}{5} h$,
  \item[(iv)] ${\rm supp}(\nabla\tilde\psi_n)\subseteq
    \big(\mathcal{U}_n\setminus \mathcal{C}_n \big)+ L
    \mathbb{Z}^2$,
  \item[(v)] $|\nabla \tilde\psi_n(x)|\le c h^{-1}$ and
    $|{\partial^2}
\tilde\psi_n(x) /
		{\partial x_i \partial x_j}
	|\le c h^{-2}$, \, for all $x\in
    \mathbb{R}^2$,
  \item[(vi)] $\big\|\nabla\tilde\psi_n \big\|_{L^2(\mathcal{D})}
    \le c $.
  \end{enumerate}
\end{proposition}

\begin{proposition}\label{prop2}
Let ${\cal K}=\{1-K,-1,0,1,-1+K\}^2$.
The functions $\tilde\psi_n$
are nearly orthogonal in the following sense:
Suppose $\alpha,\beta\in {\cal J}$ are such that
$n=\alpha_1+(\alpha_2-1)K$ and $m=\beta_1+(\beta_2-1)K$.
Then
  \begin{enumerate}
  \item[(i)] $\displaystyle \int_{ \mathcal{D}} \tilde\psi_n(x)
    \tilde\psi_m(x)\,dx=0$,\quad for $\beta-\alpha\notin {\cal K}$,
  \item[(ii)]
$\displaystyle \int_{\mathcal{D}}
    \big(\nabla\tilde\psi_n(x)\big ) \cdot \big(\nabla
    \tilde\psi_m(x)\big)\,dx=0$,\quad
        for $\beta-\alpha\notin {\cal K}$.
  \item[(iii)] $\displaystyle \Big|\int_{\mathcal{D}}
    \tilde\psi_n(x) \tilde\psi_m(x)\,dx\Big| \le
    (h+2\epsilon)^2= \frac{36}{25}h^2$,\quad
	for $\beta-\alpha\in {\cal K}$.
  \item[(iv)] $\displaystyle
    \Big|\int_{\mathcal{D}}\big(\nabla\tilde\psi_n(x)\big ) \cdot
    \big(\nabla \tilde\psi_m(x)\big)\,dx\Big| \le c$,\quad
 for $\beta-\alpha\in {\cal K}$.
  \end{enumerate}
\end{proposition}

Let us emphasize that the constant $c$, appearing in
Proposition \ref{prop1} parts
(v) and (vi), is independent of $h$.
We are now ready to prove
the following proposition on the
trace of the covariance operator $Q$ for Wiener process $W$.

\begin{proposition}\label{traces}
Let $W(t)$ be defined as in \eqref{wiener}
for the choice of $\ell_n$ given by equations
\eqref{smoothell}.  Then
$W$ is a $[\dot{H}^1]^{2}$-valued $Q$-Brownian
motion with covariance operator $Q$ that satisfies
      \begin{equation}\label{tr-Q}
        \trace[Q]\leq \frac{36}{25}\sigma^{2} L^2
      \end{equation}
and
      \begin{equation}\label{tr-Q-V}
        \trace[A^{1/2}QA^{1/2}]\leq  c \sigma^{2} \frac{L^2}{h^2}.
      \end{equation}
\end{proposition}

\begin{proof}
Since $\rho\in C^{\infty}(\mathbb{R}^2)$
then the range of $\L_h$ is in $[\dot H^\alpha]^2$, for
every $\alpha\ge 0$, and in particular for $\alpha=1$.
Therefore, $W$ is an $[\dot H^1]^2$-valued $Q$-Brownian
motion.
From \eqref{absTRQ}, Proposition \ref{prop1} part (iii)
and Proposition \ref{prop2} part (iii) we
estimate
\begin{plain}$$\eqalign{
  \trace[Q]
	&= {\sigma^{2}\over 2}\sum_{d=1}^{2N} |\gamma_d|_H^2
     \le {\sigma^2\over 2} \sum_{d=1}^{2N} \|\ell_d\|_{L^2}^2
     = \sigma^2 \sum_{n=1}^{N} \|\tilde\psi_n
			-\langle\tilde\psi_n\rangle\|_{L^2}^2 \cr
  & = \sigma^{2}   \sum_{n=1}^N L^2
		\big( \langle \tilde\psi_n^2 \rangle
		- \langle \tilde{ \psi}_{n}\rangle^2\big)
	\le \sigma^{2} N L^2 \Big(\frac{36h^2}{25L^2}- \frac{h^4}{L^4}\Big)
\le \frac{36}{25}\sigma^{2}
  L^2.
}$$\end{plain}%
Since in the periodic case we have $\|\Pi\varphi\|_V\le \|\nabla\varphi\|_{L^2}$
for every $\varphi\in\dot H^1$, then
Similarly estimate
\begin{plain}$$\eqalign{
  \trace [A^{1/2}QA^{1/2}]
	&= {\sigma^{2}\over 2}\sum_{d=1}^{2N} \|\gamma_d\|_V^2
	\le {\sigma^{2}}\sum_{n=1}^{N} |\nabla\tilde\psi_n|_{L^2}^2\cr
	&\le c\sigma^2 N=c\sigma^2{L^2\over h^2},
}$$\end{plain}%
where Proposition \ref{prop1} part (vi) has been used in
the final inequality.
\end{proof}

\subsection{The Deterministic Navier-Stokes Equations}\label{subsecclass}
The deterministic two-dimensional incompressible
Navier-Stokes equations, subject to periodic boundary conditions, are
well-posed and possess a compact
finite-dimensional global attractor.
Specifically,
the following result can be found in \cite{constantin1988},
\cite{foias2001}, \cite{robinson2001} and \cite{temam1983}.

\begin{theorem}\label{strong}
  Let $U_{0}\in V$ and $f\in H$. Then \eqref{abstract} has a unique
  strong solution that satisfies
$$U\in C([0,T]; V)\cap L^{2}([0,T]; D(A)),\quad {\rm for\, any}\quad T>0.$$
Moreover, the solution $U$ depends continuously on
$U_{0}$ in the $V$ norm.
\end{theorem}

Let us denote by $G$ the Grashof number
\begin{equation}\label{G}
  G=\frac{|f|_H}{\nu^{2}\lambda_{1}},
\end{equation}
which is a dimensionless physical parameter.  We now give bounds on
solutions $U$ of \eqref{abstract} that will be used in our later
analysis.  With the exception of inequality \eqref{boundA} these
estimates appear in the references listed above. The improved
estimate in \eqref{boundA} is given in \cite{FJLRYZ}.

\begin{theorem}\label{Bounds} Let $T>0$, and let $G$ be the
Grashof number given by \eqref{G}. There exists
  a time $t_{0}$, which depends on $U_0$, such that for all $t\geq t_{0}$
  we have
  \begin{equation}\label{boundH}
    |U(t)|_H^{2}\leq 2\nu^{2}G^{2}\quad {\rm and}\quad
    \int_{t}^{t+T}\|U(\tau)\|_{V}^{2}d\tau\leq 2(1+T\nu\lambda_{1})\nu G^{2}.
  \end{equation}
  Furthermore, we also have
  \begin{equation}\label{boundV}
    \|U(t)\|_V^{2}\leq 2\nu^{2}\lambda_{1}G^{2}\quad {\rm and}\quad
    \int_{t}^{t+T}|AU(\tau)|_{H}^{2}d\tau\leq 2(1+T\nu\lambda_{1})\nu\lambda_{1}G^{2}.
  \end{equation}
  Moreover,
  \begin{equation}\label{boundA}
    |AU(t)|_H^{2}\leq c\nu^{2}\lambda_{1}^2(1+G)^4.
  \end{equation}
\end{theorem}

\section{The Data Assimilation Algorithm}\label{secwp}

Let $U$ be the strong solution of \eqref{abstract} given by Theorem
\ref{strong}, and let $R_{h}$ be an interpolation operator satisfying
either \eqref{R1} or \eqref{R2}.
Suppose the only knowledge we have about $U$ is from the noisy
observational measurements $R_{h}(U(t))+\xi(t)$, that have been
continuously recorded for times $t\in [0,T]$.
Our goal in this section is to show that the data assimilation
algorithm given by equations \eqref{abapprox} for computing
the approximating solution $u$ are well posed.

The proof combines the well-posedness results
for the noise-free data data assimilation equations \eqref{noerror},
appearing in \cite{olson2013},
with techniques from \cite{flandoli1994}. Similar results can be found in  \cite{daprato1992} for stochastically forced partial differential equations.
Namely, we have the following two theorems.

\begin{theorem}\label{existsH}
  Suppose $U$ is the strong solution of \eqref{abstract}
	given by Theorem~\ref{strong},
  where $U_{0}\in V$ and $f\in H$. Moreover, assume
  $R_{h}\colon [\dot H^1]^2\to [\dot L^{2}]^2$
	satisfies~\eqref{R1} and that $2\mu c_1h^{2}\le \nu$.  Then
  for any $u_0\in H$ and $T>0$, there exists a unique stochastic
  process solution $u$ of equation \eqref{abapprox} in the following
sense: $\mathbb{P}$-a.s.
$$u\in C([0,T]; H)\cap L^{2}([0,T]; V)$$
and
\begin{align}\label{var}
\langle u(t),\varphi\rangle &+\int_{0}^{t} \langle u(\tau), A\varphi\rangle d\tau
-\int_{0}^{t} \langle B(u(\tau), \varphi), u(\tau)\rangle d\tau=\langle u_{0},\varphi\rangle\nonumber \\
&+\int_{0}^{t} \langle f(\tau),\varphi\rangle d\tau -\mu \int_{0}^{t} \langle  R_{h} (u(\tau)-U(\tau)), \varphi\rangle d\tau
+\langle W(t),\varphi\rangle
\end{align}
for all $t\in [0,T]$ and for all $\varphi\in D(A)$.
Moreover,
\begin{equation}\label{u-itoH}
  \E\left(\sup_{0\leq t\leq T}|u(t)|_{H}^{2}
	+\nu\int_{0}^{T}\|u(t)\|_{V}^{2}dt \right)<\infty.
\end{equation}
\end{theorem}

\begin{theorem}\label{existsV}
  Suppose $U$ is the strong solution of \eqref{abstract}
	given by Theorem~\ref{strong},
  where $U_{0}\in V$ and $f\in H$. Moreover, assume
  $R_{h}\colon [\dot H^2]^2\to [\dot H^{1}]^2$
	satisfies~\eqref{R2} and that $2\mu h^2\max (c_1,\sqrt{c_2})\le \nu$.
Then
  for any $u_0\in V$ and $T>0$, the stochastic process solution
of equation \eqref{abapprox}, given in the previous theorem
is such that
$\mathbb{P}$-a.s.
$$u\in C([0,T]; V)\cap L^{2}([0,T]; D(A))$$
and
\begin{equation}\label{u-itoV}
  \E\left(\sup_{0\leq t\leq T}\|u(t)\|_{V}^{2}
	+\nu\int_{0}^{T}|Au(t)|_{H}^{2}dt \right)<\infty.
\end{equation}
\end{theorem}

\begin{proof}[Proof of Theorem \ref{existsH}]
The proof of this theorem is based on a pathwise argument.
We proceed along the lines of \cite{flandoli1994} in which a
similar proof appears except without the function $U$ and
the additional linear term.
Consider the auxiliary process $z$ which satisfies
\begin{equation}\label{zeq}
  dz+\nu Az dt= \mu dW,\qquad z(0)=0.
\end{equation}
It is known, see \cite{daprato1992}, that
$$
	z(t)=\mu \int_{0}^t e^{-\nu A(t-\tau)} dW(\tau)
$$
is a stationary $D(A^{1/2})$-valued
ergodic solution to \eqref{zeq}
with continuous trajectories.
In particular, we have
\begin{plain}$$\eqalign{
	\E \|z(t)\|_{D(A^{1/2})}^2
	\le{\mu^2\sigma^2\over 2\nu} \trace[Q].
}$$\end{plain}%
This estimate may be obtained by writing
$$
	z=\sum_{j=1}^\infty z_j e_j\quad
\hbox{and}\quad
	W=\sum_{j=1}^\infty W_j e_j
	=\sum_{j=1}^\infty \bigg(\sum_{d=1}^D \gamma_{d,j} b_d\bigg) e_j
$$
where $z_j(t)=\langle z(t),e_j\rangle$
and $\gamma_{d,j}=\langle\gamma_d,e_j\rangle$.

Then,
$$
	z_j(t)=\mu \int_{0}^t e^{-\nu\lambda_j (t-\tau)} d W_j(\tau)
		=\mu \sum_{d=1}^D \gamma_{d,j}\int_{0}^t e^{-\nu\lambda_j(t-\tau)} db_d(\tau).
$$
Using the independence of the $b_d$'s and the It\^o isometry, it follows that
\begin{plain}$$\eqalign{
	\E |z_j(t)|^2
	&=\mu^2\sum_{d=1}^D \gamma_{d,j}^2 \E
		\bigg|\int_{0}^t e^{-\nu\lambda_j(t-\tau)} db_d(\tau)\bigg|^2\cr
	&={\mu^2\sigma^2\over 2}\sum_{d=1}^{D} \gamma_{d,j}^2
		\int_{0}^t e^{-2\nu\lambda_j(t-\tau)}d\tau
	\le {\mu^2\sigma^2\over 4\nu\lambda_j}\sum_{d=1}^D \gamma_{d,j}^2.
}$$\end{plain}%
Therefore, provided $2\alpha\le 1$ we have
\begin{plain}$$\eqalign{
	\E \|z\|_{D(A^\alpha)}^2 &= \sum_{j=1}^\infty \lambda_j^{2\alpha}
		\E |z_k|^2\cr
	&\le {\mu^2\sigma^2\over 4\nu} \sum_{d=1}^D \sum_{j=1}^\infty
		{1\over \lambda_j^{1-2\alpha}}\gamma_{d,j}^2
	\le{\mu^2
	\over 2\nu\lambda_1^{1-2\alpha}} \trace[Q].
}$$\end{plain}%

Now, using the change of variable $\tilde{u}=u-z$,
we find that $\tilde{u}$ is
solution of the (random) differential equation
\begin{equation}\label{random}
  \frac{d}{dt}\tilde{u}+\nu A\tilde{u}+B(\tilde u+z, \tilde u+z)+\mu \Pi\, R_{h}(\tilde u+z)=\tilde{f},
\end{equation}
where $\tilde{f}=f+\mu \Pi\, R_{h}(U)$ and $\tilde u(0)=\tilde u_0=u_0$.

Theorem \ref{strong} implies that $U\in C([0,T]; V)$. Hence, using
\eqref{R1} and the Poincar\'e inequality
$$
  |\Pi\, R_{h}(U)|_{H}
	\leq \|U-R_{h}(U)\|_{L^2}+|U|_{H}
	\leq \big(\sqrt{c_1}h+\lambda_{1}^{-1/2}\big)\|U\|_{V}
$$
which implies that $\Pi\, R_{h}(U)\in C([0,T]; H)$.
Therefore $\tilde f\in C([0,T];H)$.

For every
$\omega\in\Omega$, there exists a unique weak solution $\tilde{u}$ of
equation \eqref{random} and it depends continuously, in $C([0,T];
H)\cap L^{2}(0,T; V)$ norms, for any given $T>0$, on the initial
condition $\tilde u_0=u_0$ in $H$. A full rigorous proof of this
statement is very long, but at the same time it is very
classical. Similar proofs are detailed, for instance in
\cite{flandoli1994} for the stochastically forced Navier--Stokes
equations and in \cite{constantin1988} or
\cite{temam2001} in the case of the classical Navier-Stokes equations
({\it i.e.}, when $z=0$). The rigorous proof is based on the
Galerkin approximation procedure and then passing to the limit
using the appropriate compactness theorems.
We state the necessary {\it a-priori\/} estimates here.

Take the inner product of equation \eqref{random} by $\tilde{u}$
\[
\Big\langle \frac{d\tilde{u}}{dt},\tilde{u}\Big\rangle+\nu\langle
A\tilde {u},\tilde{u} \rangle=- \langle B(\tilde{u}+z,
\tilde{u}+z),\tilde{u} \rangle -\mu \langle \Pi\,
R_{h}(\tilde{u}+z),\tilde{u}\rangle + \langle
\tilde{f},\tilde{u}\rangle.\]
Using Lemma \ref{B2lem} and Young's inequality, we get that
\begin{align}
  |\langle B (\tilde{u}+z,\tilde{u}+z),\tilde{u}\rangle| &
  =|\langle B(\tilde{u},\tilde{u}),z\rangle
  +\langle B(z,\tilde{u}),z \rangle|\nonumber\\
  & \leq C_{\rm L}\left( \|\tilde{u}\|_{L^4}^{2}\Vert z\Vert_{V}+\Vert
    \tilde{u}\Vert_{V}\|z\|_{L^4}^{2}\right)  \nonumber\\
  & \leq \frac{\nu}{2}\|\tilde{u}\|_{V}^{2}
  +\frac{C_{\rm L}^{2}}{\nu}\left(|\tilde{u}|_{H}^{2}
    \|z\|_{V}^{2}+|z|_{H}^{2}\|z\|_{V}^{2}\right).
\end{align}
For the other term we apply the Cauchy-Schwarz, Young's
and Poincar\'e
inequalities along with the approximation
property \eqref{R1} to obtain
\begin{plain}\begin{align*}
    -\mu \langle \Pi\, R_{h}(\tilde{u}+z),\tilde{u}\rangle &
    =-\mu \langle \, R_{h}(z),\tilde{u}\rangle
    -\mu \langle \, R_{h}(\tilde{u}),\tilde{u} \rangle\\
    &\leq \mu |z-R_{h}(z)|_{H}|\tilde{u}|_{H} +\mu
    |z|_{H}|\tilde{u}|_{H} \\
	&\qquad\qquad+\mu \langle \tilde{u}-\,
    R_{h}(\tilde{u}),\tilde{u} \rangle
    -\mu|\tilde{u}|_{H}^{2}\\
    &\leq \frac{\mu}{2}|\tilde{u}|_{H}^{2}
    +{\frac{\mu}{2}}\left(|z|_{H}^{2}+ |z-R_{h}(z)|^{2}_{H}\right)\\
    &\qquad\qquad
		+\frac{\mu}{2}|\tilde{u}-\, R_{h}(\tilde{u})|_{H}^{2}
    +\frac{\mu}{2}|\tilde{u}|_{H}^{2}-\mu |\tilde{u}|_{H}^{2}\\
    &\leq
   { \frac{\mu}{2}}\left(\lambda_{1}^{-1}+c_1h^{2}\right)\|z\|_{V}^{2}
    +\frac{c_1h^{2}\mu}{2}\|\tilde{u}\|_{V}^{2}.
  \end{align*}\end{plain}%
Also,
\begin{plain}$$
	\langle\tilde f,\tilde u\rangle
	\le|\tilde f|_H|\tilde u|_H
	\le\lambda_1^{-1/2}|\tilde f|_H\|\tilde u\|_V
	\le{1\over\nu\lambda_1}|\tilde f|_H^2+{\nu\over4}\|\tilde u\|_V^2.
$$\end{plain}%
Hence, since we chose $h$ and $\mu$ such that $\nu\ge 2c_1h^{2}\mu$, we get
that
\begin{plain}\begin{align*}\label{v-intermediate}
  \frac{1}{2}\frac{d}{dt}
  |\tilde{u}|_{H}^{2}+\frac{\nu}{2}\|\tilde
  {u}\|_{V}^{2}&\leq \frac{\mu}{2}
	\left(\lambda_{1}^{-1}+c_1h^{2}\right)\|z\|_{V}^{2}\\
  &+\frac{C_{\rm L}^{2}}{\nu}\left(|\tilde{u}|_{H}^{2}
    \|z\|_{V}^{2}+|z|_{H}^{2}\|z\|_{V}^{2}\right)
  +\frac{1}{\nu\lambda_1}|\tilde{f}|_{H}^{2}.
\end{align*}\end{plain}%
From the above we have
\begin{plain}$$\eqalign{
	{d\over dt}|\tilde u|_H^2
	+\nu\|\tilde u\|_V^2
	\le {\mu}\big(\lambda_1^{-1}+c_1h^2\big)\|z\|_V^2
		+{2\over\nu\lambda_1}|\tilde f|_H^2\cr
	+
{2C_{\rm L}^2\over\nu}\|z\|_V^2|\tilde u|_H^2
		+
{2C_{\rm L}^2\over\nu}|z|_H^2\|z\|_V^2
.
}$$\end{plain}%
Since $\tilde f$ and $z$ are in $C([0,T];V)$, then by
Gronwall's Lemma and the previous
estimates on $z$ and $\tilde{f}$ to get that
\begin{equation*}
  \sup_{t \in [0,T]} |\tilde u(t)|_{H}^{2}
		\leq C,\quad {\rm and}\quad
  \int_{0}^{T}\|\tilde {u}(\tau)\|_{V}^{2} d\tau\leq C.
\end{equation*}
Since, $u=\tilde{u}+z$, we deduce from the properties of
$z$ that $\mathbb{P}$-a.s.
$$u\in C([0,T]; H)\bigcap L^{2}([0,T]; V).$$
The rigorous justification to the fact
that the process $u$ is adapted follows from the
limiting procedure of adapted processes (via the Galerkin
approximation).

Let us sketch the proof of \eqref{u-itoH}. As usual, these calculations
are performed in a first step on the Galerkin approximation and then
in a second step the estimates for the solution $u$ are obtained by a
limiting procedure. But for simplicity, we only sketch them for $u$. For
similar estimates, see \cite{bessaih-ferrario2014, millet2010}

Using It\^o formula to $|u(t)|_{H}^2$, where $u(t)$ is solution of
\eqref{abapprox}, we get that
\begin{equation*}
d|u(t)|_{H}^2=2\langle u(t), du(t)\rangle +\mu^2\trace[Q].
\end{equation*}
Then, using  assumption \eqref{B} and integrating over $(0,t)$, we get
\begin{align}\label{u}
\sup_{0\leq t\leq T}|u(t)|_{H}^2&+2\nu\int_{0}^{T}\|u(\tau)\|_{V}^2d\tau
	=  |u_{0}|_{H}^2 +
 \int_{0}^{T} \langle f,  u(\tau) \rangle d\tau\nonumber\\
&-2\mu\int_{0}^{T}\langle R_{h}(u(\tau)-U(\tau)), u(\tau)\rangle d\tau\nonumber\\
& +2\mu\sup_{0\leq t\leq T}\int_{0}^{t} \langle u(\tau), dW(\tau) \rangle
	+\mu^2 T \trace[Q].
\end{align}

Using the Burkholder-Gundy-Davis inequality (cf. \cite{daprato1992}) on the martingale  term in the
right-hand side of \eqref{u}
\begin{align*}
2\mu \E\left( \sup_{0\leq t\leq T}\int_{0}^{t} \langle u(\tau), dW(\tau) \rangle \right)
&\leq 2\mu \sqrt{\trace[Q]}\E \sqrt{ \int_{0}^{t} |u(\tau)|_{H}^2 d\tau}\\
&\leq 2\mu \E\sup_{0\leq t\leq T}|u(t)|_{H} \sqrt{T\trace[Q]}\\
&\leq \frac{1}{2}\E\sup_{0\leq t\leq T}|u(t)|_{H}^2+ 2\mu^2T \trace[Q].
\end{align*}

On the other hand, using the approximation \eqref{R1} and the Poincar\'e
inequality, we get that
\begin{align*}
-2\mu\langle R_{h}(u&-U), u\rangle \\
&\leq 2\mu |u-R_{h}(u)|_{L^2} |u|_H-2\mu |u|_{H}^2+2\mu |R_{h}(U)|_{L^2}|u|_H\\
&\leq 2\mu |u-R_{h}(u)|_{L^2}^2-\mu |u|_{H}^2 +2\mu |R_{h}(U)|_{L^2}^2\\
&\leq 2\mu c_{1}h^2 \|u\|_V^2-\mu |u|_{H}^2
	+2\mu\left( |U-R_{h}(U)|_{L^2}^2+|U|_H^2\right)\\
&\leq 2\mu c_{1}h^2 \|u\|^2-\mu |u|_{H}^2 +2\mu(c_{2}h^2+\lambda_{1}^{-1})
	\|U\|_V^2.
\end{align*}
Since $2\mu c_{1}h^2\le \nu$, combining the previous estimates with
the Gronwall lemma one obtains
$$\E \left(\sup_{0\leq t\leq T}|u(t)|_{H}^2\right) \leq
C\left(|u_{0}|_{H}, T, \mu, \nu, h, \lambda_{1}, \trace[Q], \|U\|_{V}\right).$$
Using again this estimate in \eqref{u} finishes the proof.
\end{proof}

\begin{proof}[Proof of Theorem \ref{existsV}]
Similar arguments as used in Theorem \ref{existsH} may be used
to prove this theorem.
\end{proof}

\section{Main Results}\label{secmain}

Our goal is to prove that $u$, the solution of \eqref{abapprox},  approximates
the true solution $U$ of \eqref{abstract}, when $t\to \infty$, to within some
tolerance depending on the error in the observations.
Upon setting $v=U-u$ this is equivalent to showing that $v$
is small.
Note that from \eqref{abstract} and \eqref{abapprox} the evolution of $v$ is governed by the equation
\begin{plain}\begin{equation}\label{v}\eqalign{
      dv+\big[\nu A v+B(U,v)+B(v,U)-&B(v,v)\big]dt\cr
		&=-\mu \Pi\, R_{h}(v) dt+\mu \Pi \, dW,
}\end{equation}\end{plain}%
with $v_0 \in V$ is chosen arbitrary.
We now look for conditions on $h$ and $\mu$
such that the feedback term, $R_h$, on the right-hand side of this
equation, which is stabilizing the coarse scales, together
with the viscous term, which is stabilizing the fine scales,
controls the growth of $v$, which is due to the
unstable nature of this kind of nonlinear dynamical system.

Section \ref{subsecm1} studies interpolant
observables which satisfy \eqref{R1} and, in particular, those
 come from finite volume elements.
Section \ref{subsecm2} studies interpolant
observables which satisfy \eqref{R2} and, in particular, those
which come from nodal observations.

\subsection{Observations of Volume Elements}\label{subsecm1}

This section first proves a general theorem on interpolant
observables that satisfy \eqref{R1}.  This result is then
applied to obtain explicit estimates when the observational
measurements arise, for example, from volume elements.
The same result holds for other kinds of
observables that satisfy \eqref{R1}, such as
Fourier modes and the interpolants investigated in \cite{Holst-Titi}.

\begin{theorem}\label{main1} Assume that $U$ is a strong solution of
  \eqref{abstract}, that $R_{h}$
  satisfies assumption \eqref{R1} and that $W$ is a
$[\dot L^2]^2$-valued $Q$-Brownian motion.
Assume that $\mu$ is large enough, and $h$ is small enough
such that
\begin{plain}$$
	{1\over h^2}\ge {2c_1\mu\over\nu}\ge 8 c_1 C_{\rm L}^2\lambda_1 G^2.
$$\end{plain}%
where $c_{1},\ C_{L}$ are respectively
given in \eqref{R1} and \eqref{inter}.
Then
the solution $u$ of \eqref{abapprox} given by Theorem
	\ref{existsH} satisfies
\begin{equation*}
  \limsup_{t\rightarrow\infty}\E(|u(t)-U(t)|_{H}^{2})
	\le \mu\trace[Q].
\end{equation*}
Moreover,
\begin{plain}\begin{equation}\label{vavg}
	\limsup_{t\to\infty}
{\nu\over T}\int_{t}^{t+T} \E (\|u(\tau)-U(t)\|_V^2)d\tau
	\le \Big({1\over T}+\mu\Big)\mu\trace[Q].
\end{equation}\end{plain}%
\end{theorem}

\begin{proof}
In this proof we focus on the time interval $[t_0, \infty)$, where
$t_0$ is given in Theorem \ref{Bounds}. Using the It\^{o} formula on
$|v(t)|_{H}^{2}$ we obtain
$$
		d|v|_{H}^{2}=2\langle v,dv\rangle+
			\mu^2 \trace[Q]\,dt.
$$
Substituting for $dv$ and applying the
orthogonality property \eqref{B} yields
\begin{plain}\begin{equation}\label{u-T1}\eqalign{
	d|v|_{H}^{2}+2\nu\|v\|_{V}^{2} dt
		=&-2\langle v,B(v, U)\rangle dt
		-2\mu\langle v,R_{h}v\rangle dt\cr
		&+2\mu \langle v,dW \rangle +\mu^2\trace[Q]\,dt.
}\end{equation}\end{plain}%
Estimate the first two terms of the right-hand side as follows:
using inequality~\eqref{bilinest} and Young's inequality
\begin{equation}\label{term1a}
    -2\langle v,B(v, U)\rangle
	\le \nu\|v\|_{V}^{2}
    	+\frac{C_{\rm L}^2}{\nu}|v|_{H}^{2}\|U\|_{V}^{2}.
\end{equation}
Using Young's inequality, the interpolation inequality \eqref{R1} and
the assumption that $2c_1\mu h^2\le\nu$ we obtain
\begin{plain}\begin{equation}\label{term2a}
    \begin{split}
      -2\mu\langle v,R_{h}v\rangle
			&=2\mu\langle v,v- R_{h} (v)\rangle-2\mu |v|_{H}^{2}\\
      &\le 2\mu |v-R_{h}(v)|_{L^2}^2-{3\mu\over 2}|v|_{H}^{2}
      \le \nu\|v\|_{V}^{2}-{3\mu\over 2} |v|_{H}^{2}.
    \end{split}
\end{equation}\end{plain}%
Therefore,
\begin{plain}$$
	d|v|_H^2 +
	\Big({3\mu\over 2}-{C_{\rm L}^2\over \nu} \|U\|_V^2\Big) |v|_H^2\,dt
		\le
      2\mu \langle v,dW\rangle +\mu^2\trace[Q]\,dt.
$$\end{plain}%
Since $t \ge t_0$ then  inequality \eqref{boundV} and the hypothesis
$2 C_{\rm L}^2\nu\lambda_1 G^2\le \mu/2$ yield
\begin{plain}$$ {C_{\rm L}^2\over\nu} \|U(t)\|_V^2 \le
	2C_{\rm L}^2\nu \lambda_1 G^2 \le {\mu\over 2},$$\end{plain}%
which implies
\begin{equation}\label{pathwise}
	d|v|_H^2 + \mu |v|_H^2\,dt
		\le
      2\mu \langle v,dW\rangle  +\mu^2\trace[Q]\,dt,
\end{equation}
for all $t \ge t_0$.
Now integrating over $(t_{0}, t)$, then taking the expected
value and
using Gronwall's lemma, we obtain
$$
	\E (|v(t)|_H^2)\le \E (|v(t_0)|_H^2) e^{-\mu (t-t_0)}
		+\mu \trace[Q],
	\quad\hbox{for all}\quad
		t\ge t_0.
$$
Thus, it follows that
$$
	\limsup_{t\to\infty} \E(|v(t)|_H^2)\le \mu\trace[Q].
$$

To obtain \eqref{vavg}, we estimate the terms in \eqref{term1a}
and \eqref{term2a} using Young's inequality in
a slightly different way.  In particular,
\begin{plain}$$
	-2\langle v,B(u,U)\rangle \le
		{\nu\over 2} \|v\|_V^2+
		{2 C_L^2\over \nu} |v|_H^2 \|U\|_V^2
$$\end{plain}%
and
\begin{plain}$$
	-2\mu\langle v,R_h v\rangle \le
		{\nu\over 2} \|v\|_V^2-\mu |v|_H^2.
$$\end{plain}%
Therefore,
$$
	d|v|_H^2+\nu \|v\|_V^2 dt \le 2\mu\langle v,dW\rangle
		+ \mu^2 \trace[Q] dt.
$$
Taking expected values and integrating from $t$ to $t+T$
yields
$$
	\E (|v(t+T)|_H^2)+\nu\int_{t}^{t+T} \E (\|v(\tau)\|_V^2)d\tau
	\le \E (|v(t)|_H^2) + \mu^2T\trace[Q].
$$
Therefore
$$
	\limsup_{t\to\infty}
\nu\int_{t}^{t+T} \E (\|v(\tau)\|_V^2) d\tau
	\le \mu\trace[Q]+\mu^2 T\trace[Q],
$$
from which \eqref{vavg} immediately follows.
\end{proof}

Theorem \ref{main1} applies to observations of the
volume elements given by \eqref{volumes}.
We now state and prove a corollary on finite volume elements
that gives explicit estimates on the
how well $u$ approximates $U$ over time.

\begin{corollary}\label{cor1}
Suppose that the observational measurements
are given by finite volume elements \eqref{volumes}
plus a noise term of the form \eqref{noise}, where
each $b_d$ is an independent one-dimensional Brownian motion
with variance $\sigma^2/2$.
Interpolate these noisy observations using \eqref{genL}
where $\ell_d$ are given by \eqref{stepell}.
Let $\mu=4 C_{\rm L}^2\nu\lambda_1 G^2$ and choose $N=K^2$
large enough such that
$$h=L/K \le \sqrt{\nu/(2c_1\mu)}.$$
Then
the solution $u$ to \eqref{abapprox} satisfies
\begin{equation*}
  \limsup_{t\rightarrow\infty}\E(|u(t)-U(t)|_{H}^{2})
	\le \kappa_1 \nu G^2 \sigma^2
\end{equation*}
and
\begin{plain}$$
\limsup_{t\to\infty}
{\nu\over T}\int_{t}^{t+T} \E (\|u(\tau)-U(\tau)\|_V^2) d\tau
    \le \Big({1\over T}
	+{\kappa_1\nu G^2\over L^2}\Big)\kappa_1 \nu G^2 \sigma^2
$$\end{plain}%
where $\kappa_1=16\pi^2 C_{\rm L}^2$
is an absolute constant.
\end{corollary}

\begin{proof}
By Proposition \ref{traceH} we obtain
$$
	\mu\trace[Q]
		\le \mu \sigma^2 L^2
		\le 4 C_{\rm L}^2\nu\lambda_1 G^2 \sigma^2 L^2
		=\kappa_1\nu G^2\sigma^2.
$$
where $\kappa_1=16\pi^2 C_{\rm L}^2$. Similarly
\begin{plain}$$
	\Big({\mu\over T}+\mu^2\Big)\trace[Q]\le
		\Big({\kappa_1\nu G^2\over T}+{\kappa_1^2\nu^2G^4\over L^2}\Big)
		\sigma^2.
$$\end{plain}%
Since the choices of $\mu$ and $h$ given in the corollary
satisfy the hypothesis
of Theorem \ref{main1}, the corollary is thus proved.
\end{proof}

We remark that the upper bound on the error in the
approximating solution given by Corollary \ref{cor1} is
independent of $h$.
In particular, as we take the observation density
finer and finer there is no improvement in the quality of our
approximation.
This is not surprising, since increasing the resolution of
the observations did not lead to any decrease in the size of the
measurement errors present in the interpolant observables
$\tilde R_h$ given by \eqref{genL}.
We remedy this defect with

\begin{corollary}\label{cor2}
Suppose that the observational measurements
are given by finite volume elements \eqref{volumes}
plus a noise term of the form \eqref{noise} where
each $b_d$ is an independent one-dimensional Brownian motion
with variance $\sigma^2/2$.
Let $\mu$ be as in Corollary \ref{cor1} and $\epsilon\in(0,1)$.
Then, there exists an interpolant observable based on
volume elements with observation density $h$ such that
\begin{plain}\begin{equation}\label{hover}
	{\kappa_2 G^2\over L^2}
\le {\epsilon\over h^2}
	\le {\max(\epsilon,16\kappa_2 G^2)\over L^2}
\end{equation}\end{plain}%
where $\kappa_2=32\pi^2c_1 C_{\rm L}^2$ is an absolute constant
and for which
\begin{equation*}
  \limsup_{t\rightarrow\infty}\E(|u(t)-U(t)|_{H}^{2})
	\le \mu \sigma^2 L^2 \epsilon
\end{equation*}
and
\begin{plain}\begin{equation}
    \limsup_{t\to\infty}
{\nu\over T}\int_{t}^{t+T} \E (\|u(\tau)-U(\tau)\|_V^2) d\tau
    \le \Big({1\over T}+\mu\Big)\mu\sigma^2 L^2\epsilon.
\end{equation}\end{plain}%
\end{corollary}

\begin{proof}
If $\sqrt{\nu/(2c_1\mu)}\ge L$ then we may take $h=L$
in Theorem \ref{main1}.  In this case $U$ is a steady
state, and consequently no observational data is needed
to accurately recover $U$.  Otherwise,
let $M=K_2^2$ where $K_2\ge 2$ is the unique integer such that
$$h'=L/K_2\le \sqrt{\nu/(2c_1\mu)}<L/(K_2-1).$$
Let $Q'_{m}$ be squares with sides of length $h'$
where $m=1,2,\ldots,M$
defined in a similar way as \eqref{squares}.
Choose $h=h'/q$ where $q$ is the unique integer satisfying
\begin{plain}\begin{equation}\label{choiceq}
q^2\ge \epsilon^{-1}>(q-1)^2.
\end{equation}\end{plain}%
With these choices of $K_2$ and $q$ we have
$$
	\sqrt {2c_1\mu/\nu}\le 1/h'=K_2/L\le 2(K_2-1)/L<2\sqrt{2c_1\mu/\nu}
$$
and
$$
	\epsilon^{-1/2}\le q=(q-1)+1\le \epsilon^{-1/2}+1\le 2\epsilon^{-1/2}.
$$
Therefore
$$
\sqrt{\nu\epsilon/(32 c_1\mu)}\le
	h=h'/q\le
\sqrt{\nu\epsilon/(2 c_1\mu)}
$$
from which \eqref{hover} follows.

Let $Q_n$ be the squares with sides of length $h$
where $n=1,2,\ldots,N$ and $N=L^2/h^2=q^2M$.
The smaller squares $Q_n$ fit inside the
larger squares $Q'_m$ and each larger square is the
union over $q^2$ of the smaller squares.

Define the averaging operator
${\cal A}\colon \mathbb{R}^{2N}\to \mathbb{R}^{2M}$ by
\begin{plain}$$
	\left[\matrix{\smallskip
		{\cal A}(\varphi)_{2m-1}\cr
		{\cal A}(\varphi)_{2m}
	}\right]
=
		{1\over q^2}
		\sum_{Q_n\subseteq Q'_m}
	\left[\matrix{\smallskip
		\varphi_{2n-1}\cr
		\varphi_{2n}
	}\right],
$$\end{plain}%
for $m=1,2,\ldots,M$.
We note that ${\cal O}_{h'}={\cal A}\circ{\cal O}_h$, where
${\cal O}_h$ are the noise-free observations of volume elements
given in \eqref{volumes}, for the $Q_n$ and
${\cal O}_{h'}$ are
the analogous observations for $Q'_m$.

Let $\tilde {\cal O}_h(U(t))$ be the noisy observations defined
by \eqref{noisyO}, where ${\cal E}(t)$ is given by \eqref{noise}.
It follows that $ {\cal A}\circ\tilde {\cal O}_h(U(t))=
	{\cal O}_{h'}(U(t))+{\cal F}(t)$,
where
$$
	{\cal F}(t)dt=(d\beta_1(t),d\beta_2(t),\ldots,d\beta_{2M}(t))
$$
and $\beta_k$ are one-dimensional independent Brownian motions
such that
\begin{plain}$$\E(\beta_k(t))=0
\quad\hbox{and}\quad
\E(\beta_k(t)^2)=t{\sigma^2\over 2q^2},$$\end{plain}%
for $k=1,2,\ldots,2M$.
Therefore, by taking averages of volume elements, also called
spatial oversampling,
we have reduced the variance in the noise term of the
measurements.
In particular, from \eqref{choiceq} the noise term
is now equivalent to a $[\dot L^2]^2$-valued
$Q'$-Brownian motion
with $\trace[Q']\le\sigma^2L^2/q^2\le \sigma^2 L^2 \epsilon$.
We now define the interpolant observable
$${\cal R}_{h'}= \L_{h'}\circ {\cal A}\circ {\cal O}_h.$$
Since ${\cal R}_{h'}$ satisfies \eqref{R1} with the
same constants as before, then
applying Theorem \ref{main1} now completes the proof.
\end{proof}

\subsection{Observations of Nodal Values}\label{subsecm2}

This section first proves a general theorem on interpolant observables
which satisfy \eqref{R2}.  This result is then applied to obtain
explicit estimates when the observational measurements arise
from nodal measurements.

Our proof follows the general strategy of the non-stochastic case
treated in \cite{olson2003} with modifications as was done in the
proof of Theorem \ref{main1} above to account
for the stochastic terms which arise from the stochastic
errors.
In particular, we shall make use of the following inequality
which can be found in \cite{olson2003}.

\begin{lemma}
  \label{minlog}
  Let $\varphi(r)=r-\eta(1+\log r)$ where $\eta>0$.  Then
$$
\min\{\varphi(r):r\ge 1\}\ge-\eta\log \eta.$$
\end{lemma}

We commence with the proof of

\begin{theorem}\label{main2} Assume that $U$ is a strong solution of
  \eqref{abstract}, that $R_{h}$ satisfies assumption \eqref{R2}
  and that $W$ is a $[\dot H^1]^2$-valued $Q$-Brownian motion.
Assume that $\mu$ is large enough and that $h$ is small enough such that
\begin{plain}
$$
 \mu \ge 2 \nu  \lambda_1 G J, \quad \hbox{and} \quad
	\nu \ge 2 c_3 h^2 \mu
$$\end{plain}%
where $c_3=\max(c_1,\sqrt c_2)$ and
	$J=2C_{\rm B}(2+\log 2C_{\rm B}c^{1/2})(1+\log(1+G))$.
Then
$$
	\limsup_{t\to\infty} \E (\|u-U\|_V^2)\le
4\mu\exp(\nu\lambda_1 G^2 J^2/\mu) \trace[A^{1/2}Q A^{1/2}]$$
and
\begin{plain}$$\eqalign{
	&\limsup_{t\to\infty} {\nu\over T}
		\int_t^{t+T}\E(|Au(\tau)-AU(\tau)|_H^2) d\tau\cr
	&\qquad\le
\bigg\{8\exp(\nu\lambda_1G^2J^2/\mu)
	\Big\{{\mu\over T}+4J^2\Big({1\over T}+\nu\lambda_1\Big)
		\nu\lambda_1 G^2\Big\}
	+\mu^2\bigg\}\cr
			&\qquad\qquad\times 2\trace[A^{1/2}QA^{1/2}].
}$$\end{plain}%
\end{theorem}

\begin{proof}
We focus on the interval $[t_0, \infty)$, where $t_0$ is given in Theorem \ref{Bounds}. Using the It\^{o} formula on $\|v(t)\|_{V}$ we obtain
$$
	d\|v\|_V^2=2\vl v,dv\vr + \mu^2 \trace[A^{1/2} Q A^{1/2}] dt.
$$
For notational convenience we shall write $\Sigma=\trace[A^{1/2} Q A^{1/2}]$
throughout the rest of this proof.
Substituting for $dv$ and applying \eqref{B1} and \eqref{B2} yields
\begin{plain}$$\eqalign{
	d\|v\|_V^2+2\nu |Av|_H^2 dt=
		&2\langle AU,B(v,v)\rangle dt -2\mu \langle Av,R_h(v)\rangle dt\cr
		&+2\mu\langle Av,dW\rangle
		 + \mu^2 \Sigma dt.
}$$\end{plain}%

The Br\'ezis--Gallouet inequality \eqref{brezis}
implies
\begin{plain}\begin{equation*}
  \begin{split}
    |\langle AU,B(v, v)\rangle|&\leq \|v\|_{\infty}\|v\|_{V} |AU|_{H}\\
    &\leq C_{\rm B} \|v\|_V^{2}\bigg\{ 1+\log {|Av|_{H}^2\over \lambda_1
      \|v\|_{V}^2}\bigg\}|AU|_{H}
  \end{split}
\end{equation*}\end{plain}%
and the assumption
$2\mu\max(c_1,\sqrt{c_2}) h^2\le\nu$ along with \eqref{R2}
and Young's inequality implies
\begin{plain}$$\eqalign{
-2\mu \langle Av,R_h(v)\rangle
	&=2\mu \langle Av,v-R_h(v)\rangle - 2\mu \|v\|_V^2\cr
	&\le {2\mu^2\over\nu} |v-R_h(v)|_{L^2}^2
	+{\nu\over 2} |A v|_H^2-2\mu\|v\|_V^2\cr
	&\le \nu |Av|_H^2-\mu\|v\|_V^2.
}$$\end{plain}%
Therefore
\begin{plain}\begin{equation*}
  \begin{split}
    d\|v\|_{V}^{2}+ \bigg(\nu\lambda_{1} {|Av|_{H}^2\over \lambda_1
        \|v\|_{V}^2} &-2C_{\rm B} |AU|_{H}
	\bigg\{
      1+\log {|Av|_{H}^2\over \lambda_1 \|v\|_{V}^2}\bigg\}
		+\mu \bigg)\|v\|_{V}^{2}dt\\
    &\leq 2\mu\langle Av,d W\rangle+\mu^2 \Sigma dt.
  \end{split}
\end{equation*}\end{plain}%

Now setting
\begin{plain}$$
\eta=\frac{2C_{\rm B}|AU|_{H}}{\nu\lambda_{1}}
	\quad\hbox{and}\quad
r={|Av|_{H}^2\over \lambda_1 \|v\|_{V}^2}
$$\end{plain}%
in Lemma \ref{minlog}, and noting that $r\ge 1$, we obtain that
\begin{plain}\begin{equation*}
  \begin{split}
    d\|v\|_{V}^{2}+\bigg( \mu&-2C_{\rm B}|AU|_H\log
      {2C_{\rm B}|AU|_{H}\over \nu\lambda_1}\bigg)\|v\|_{V}^{2}dt \\
    &\leq 2\mu\langle  Av,d W\rangle
    +\mu^2\Sigma dt.
  \end{split}
\end{equation*}\end{plain}%
Since $t \ge t_0$, then by \eqref{boundA} we estimate
\begin{plain}\begin{equation*}
  2C_{\rm B}\log {2C_{\rm B}|AU|_{H}\over \nu\lambda_1}
	\le 2C_{\rm B}\log {2C_{\rm B}c^{1/2}(1+G)^2}
	\le J.
\end{equation*}\end{plain}%
Consequently,
\begin{equation*}
  d\|v\|_{V}^{2}
	+\Big\{ \mu-J|AU|_{H}\Big\}\|v\|_{V}^{2}dt
	\leq 2\mu\langle  Av,d W\rangle
  +\mu^2\Sigma dt.
\end{equation*}
Applying Young's inequality  we get
\begin{plain}$$
	d \|v\|_V^2
	+{1\over 2}\Big\{ \mu-{J^2\over\mu}|AU|_{H}^2\Big\} \|v\|_{V}^{2}dt
	\le 2\mu\langle  Av,d W\rangle +\mu^2\Sigma dt.
$$\end{plain}

Take
$t_0$ as in Theorem \ref{Bounds} and define
\begin{plain}$$
	\alpha(t)={1\over 2}\Big\{ \mu-{J^2\over\mu}|AU(t)|_{H}^2\Big\}
\quad\hbox{and}\quad
	\Psi(t)=\int_{t_0}^t \alpha(s) ds.
$$\end{plain}%
Now, integrating and taking expected value yields
$$
	\E (\|v(t)\|_V^2)\le \E (\|v(t_0)\|_V^2) e^{-\Psi(t)}
		+\mu^2\Sigma \int_{t_0}^t e^{-\Psi(t)+\Psi(\tau)} d\tau.
$$
Since $\mu> 2\nu\lambda_1GJ$ then by the estimate \eqref{boundV} we
obtain
\begin{plain}$$\eqalign{
	-\Psi(t)+\Psi(\tau)&\le -{\mu\over 2} (t-\tau) +
		{J^2\over\mu}(1+(t-\tau)\nu\lambda_1)\nu\lambda_1 G^2\cr
	&\le -{\mu\over 4} (t-\tau) + {\nu\lambda_1\over\mu}G^2 J^2.
}$$\end{plain}%
Therefore $-\Psi(t)\to -\infty$ and
\begin{plain}$$
	\int_{t_0}^t e^{-\Psi(t)+\Psi(\tau)}d\tau
		\le  {c_4\over \mu} \big(1-e^{-\mu(t-t_0)/4}\big)
	\to {c_4\over \mu},
\quad\hbox{as}\quad
	t\to\infty,
$$\end{plain}%
where $c_4=4\exp(\nu\lambda_1 G^2 J^2/\mu)$.  It follows that
$$
	\limsup_{t\to\infty} \E (\|v(t)\|_V^2)
		\le c_4\mu\Sigma
		= 4\mu\exp(\nu\lambda_1 G^2 J^2/\mu) \trace[A^{1/2}Q A^{1/2}],
$$
which is the first inequality.

To obtain the second inequality, we use
\begin{plain}$$
	\eta={4C_{\rm B} |AU|_H\over \nu\lambda_1}
\quad\hbox{and}\quad
	r={|Av|_H^2\over\lambda_1 \|v\|_V^2}
$$\end{plain}%
in Lemma \ref{minlog} to obtain
\begin{plain}\begin{equation}\label{TAofA}\eqalign{
	d\|v\|_V^2+{\nu\over 2}|Av|_H^2 dt
	+{1\over 2}
		\Big\{\mu&-{\tilde J^2\over\mu}|AU|_H^2\Big\}\|v\|_V^2 dt\cr
	&\le  2\mu\langle Av,d W\rangle dt + \mu^2\Sigma dt}
\end{equation}\end{plain}%
where
$$\tilde J=4C_{\rm B}\log 4C_{\rm B}c^{1/2}(1+G)^2\le 2 J.$$
Let $t_1\ge t_0$ be large
enough such that
$$\E(\|v(t)\|_V^2) \le 2c_4\mu\Sigma\qquad\hbox{for}\qquad t\ge t_1.$$
After integrating inequality \eqref{TAofA} and
taking expected values, we obtain for times $t>t_1$ that
\begin{plain}$$\eqalign{
	{\nu\over 2} \int_t^{t+T} \E |Av(\tau)|_H^2d\tau
		&\le 2c_4\mu \Sigma
		+c_4 \tilde J^2\Sigma \int_t^{t+T} |AU(\tau)|_H^2 d\tau
			+\mu^2 T\Sigma\cr
		&\le 2c_4\mu \Sigma
		+8c_4 J^2\Sigma (1+T\nu\lambda_1)\nu\lambda_1 G^2+\mu^2T\Sigma.
}$$\end{plain}%
This finishes the proof.
\end{proof}

The bounds on $h^{-2}$ are proportional to $G(1+\log(1+G))$ which is
similar to the deterministic case.  However, the bounds on the
expected value of $\|u-U\|_V^2$ depend exponentially on $G$.
Therefore, unless the variance in the stochastic error
represented by $\trace[A^{1/2}Q A^{1/2}]$ is very small, this
bound will be very large.  However, this exponential dependence
on $G$ may be removed by taking
$\mu=\nu\lambda_1G^2J^2$ and
$h^{-2}$ correspondingly larger.  This yields

\begin{corollary}\label{cor1main2}
Suppose $\mu=c_5\nu\lambda_1 G^2(1+\log(1+G))^2$
and $h$ is small enough such that
\begin{plain}$$
    \nu  \ge 2 c_3\mu h^2,
$$\end{plain}%
where $c_3$ is defined as in Theorem \ref{main2}
and $c_5=4 C_{\rm B}^2(2+\log2C_{\rm B}c^{1/2})^2$.
Then
$$
    \limsup_{t\to\infty} \E (\|u-U\|_V^2)\le
        4e \mu \trace[A^{1/2}Q A^{1/2}]$$
and
\begin{plain}$$\eqalign{
	&\limsup_{t\to\infty} {\nu\over T}
		\int_t^{t+T}\E(|Au(\tau)-AU(\tau)|_H^2) d\tau\cr
	&\qquad\qquad\le
\Big\{ {20\over T} + 16\nu\lambda_1+{\mu\over 2e}\Big\}
		4e\mu \trace[A^{1/2}QA^{1/2}].
}$$\end{plain}%
\end{corollary}

We apply Corollary \ref{cor1main2} to nodal observations
described by \eqref{nodes}
to obtain

\begin{corollary}\label{nodcor1}
Suppose that the observational measurements are given by
nodal observations \eqref{nodes} plus a noise term of the form
\eqref{noise}, where each $b_d$ is an independent one-dimensional
Brownian motion with variance $\sigma^2/2$.
Interpolate the noisy observations
using \eqref{genL} where $\ell_d$ are given by
\eqref{smoothell}.
Suppose that $\sqrt{\nu/(2c_3\mu)}<L$
where $\mu=c_5\nu\lambda_1G^2(1+\log(1+G))^2$ and
choose $N=K^2$
such that
$$h=L/K\le \sqrt{\nu/(2c_3\mu)}<L/(K-1).$$
Then the solution $U$ to \eqref{abapprox} satisfies
$$
	\limsup_{t\to\infty} \E(\|u(t)-U(t)\|_V^2)
		\le
		\kappa_3 \nu \lambda_1 G^4(1+\log(1+G))^4 \sigma^2
$$
and
\begin{plain}$$\eqalign{
    \limsup_{t\to\infty} {\nu\over T}
		\int_t^{t+T} \E (|Au(\tau)&-AU(\tau)|_H^2) d\tau\cr
	\le
\Big\{ {20\over T} + 16\nu\lambda_1&+{c_5\nu\lambda_1 \over 2e}
		G^2(1+\log(1+G))^2\Big\}\cr
		&\times
		\kappa_3 \nu \lambda_1 G^4(1+\log(1+G))^4 \sigma^2
}$$\end{plain}%
where $\kappa_3=128\pi^2ecc_3c_5^2$
is an absolute constant.
\end{corollary}
\begin{proof}
By Proposition \ref{traces} equation \eqref{tr-Q-V}, we obtain
\begin{plain}$$\eqalign{
	\trace[A^{1/2}Q A^{1/2}]
	&\le c\sigma^2 {L^2\over h^2}
	\le c\sigma^2 K^2
	\le 4c\sigma^2 (K-1)^2\cr
	&< 32\pi^2 cc_3c_5 G^2(1+\log(1+G))^2\sigma^2.
}$$\end{plain}%
Since $\mu$ and $h$ satisfy the hypothesis of Corollary \ref{cor1main2},
then
$$
    \limsup_{t\to\infty} \E (\|u-U\|_V^2)\le
128\pi^2ecc_3c_5^2\nu\lambda_1 G^4(1+\log(1+G))^4 \sigma^2$$
and the second inequality follows similarly.
\end{proof}

We end by noting that the oversampling argument used to reduce
the error in Corollary \ref{cor2} can also be used
with nodal measurements.  Along these lines we obtain

\begin{corollary}
Suppose that the observational measurements
are given by nodal observations \eqref{nodes}
plus a noise term of the form \eqref{noise} where
each $b_d$ is an independent one-dimensional Brownian motion
with variance $\sigma^2/2$.
Let $\mu$ be as in Corollary \ref{cor1main2} and $\epsilon\in(0,1)$.
Then, there exists an interpolant observable based on
nodal measurements with observation density $h$ such that
\begin{equation*}
	{\kappa_4 G^2(1+\log(1+G))^2\over L^2} \le {\epsilon\over h^2}
	\le {\max(\epsilon,16\kappa_4 G^2(1+\log(1+G))^2)\over L^2}
\end{equation*}
where $\kappa_4 =32\pi^2 c_3 C_{\rm B}^2(2+\log 2 C_{\rm B}c^{1/2})^2$
is an absolute constant and for which
\begin{equation*}
  \limsup_{t\rightarrow\infty}\E(\|u(t)-U(t)\|_{V}^{2})
    \le
32 e c c_3 \mu^2\nu^{-1} \sigma^2 L^2\epsilon
\end{equation*}
and
\begin{plain}$$\eqalign{
	&\limsup_{t\to\infty} {\nu\over T}
		\int_t^{t+T}\E|Au(\tau)-AU(\tau)|_H^2d\tau\cr
	&\qquad\qquad\le
\Big\{ {20\over T}+16\nu\lambda_1+{\mu\over 2e}\Big\}
32 e c c_3 \mu^2\nu^{-1} \sigma^2 L^2\epsilon.
}$$\end{plain}%
\end{corollary}
\begin{proof}
Define $h'$, $K_2$, $M$, $N$, $q$, $Q_n$ and $Q'_m$ as in the
proof of Corollary \ref{cor2} where we have taken $c_3$ in
place of $c_1$.  Let $x_n\in Q_n$ for $n=1,2,\ldots,N$.
Since the $Q_n$ are disjoint then the $x_n$ are distinct.
Inside each large square $Q'_m$ fit $q^2$
smaller squares $Q_n$
and therefore $q^2$ points $x_n$.  Denote
$$
	\{\, x_n : x_n\in Q_m\,\} = \{\, x'_{m,j}: j=1,2,\ldots,q^2\,\}.
$$
Since $x'_{m,j}\in Q_m$ for each $j=1,\ldots,q^2$, we may view ${\cal O}_h$
as a family of $q^2$ observations of
nodes ${\cal O}^j_{h'}\colon[\dot H^2]^2\to\mathbb{R}^{2M}$
given by
\begin{plain}$$
	{\cal O}^j_{h'}(\Phi)=(\varphi_{1,j},\ldots,\varphi_{2M,j})
\quad\hbox{where}\quad
	\left[\matrix{
		\varphi_{2m-1,j}\cr
		\varphi_{2m,j}}
	\right]
	=\Phi(x'_{m,j})
$$\end{plain}%
and $m=1,2,\ldots, M$.  This leads to a family of $q^2$ independent
noisy observations $\tilde {\cal O}^j_{h'}(U(t))$.  It follows
that the average of the noisy observations
$$
	\tilde {\cal O}_{h'}(U(t))
	={1\over q^2}\sum_{j=1}^{q^2} \tilde{\cal O}^j_{h'}(U(t))
	={1\over q^2}\sum_{j=1}^{q^2} {\cal O}^j_{h'}(U(t)) + {\cal F}(t)
$$
where
$$
	{\cal F}(t)dt=(d\beta_1(t),d\beta_2(t),\ldots,d\beta_{2M})
$$
and the $\beta_k$ are one-dimensional independent Brownian motions
such that $\E(\beta_k(t))=0$ and $\E(\beta_k(t)^2)=t\sigma^2/(2q^2)$
for $k=1,2,\ldots,2M$.  Therefore, just as in the case with
finite volume elements, we have reduced the variance in the noise
term by averaging.  In particular, the noise term is now
equivalent to an $[\dot H^1]^2$-valued $Q'$-Brownian motion
with
\begin{plain}$$
	\trace[A^{1/2}Q'A^{1/2}]\le c \sigma^2 \Big({L\over h'}\Big)^2{1\over q^2}
		\le 8c c_3 \mu\nu^{-1} \sigma^2 L^2\epsilon.
$$\end{plain}%
We now define the interpolant observable
$$
	{\cal R}_{h'}=
	{1\over q^2}\sum_{j=1}^{q^2} \L_{h'}\circ {\cal O}^j_{h'}(U(t)).
$$
Since ${\cal R}_{h'}$ satisfies \eqref{R2} with the same
constants as before, then applying Corollary \ref{cor1main2}
now completes the proof.
\end{proof}

\section{Conclusions}

We have shown the continuous data assimilation algorithm
proposed in~\cite{olson2013} continues to be well posed when
the observational measurements contain errors represented by
stochastic noise.
Provided the resolution of the observational data is fine enough,
we have shown that the expected value of the difference between the
approximate solution, recovered by this data
assimilation algorithm, and the exact solution is bounded by a factor
depending on the Grashof number times the variance of the noise,
asymptotically in time.
This occurs for general interpolant operator observables satisfying either one of the approximate identity properties
\eqref{R1} or \eqref{R2}, and, in particular, for
interpolant observables based on volume
elements and nodal measurements.

In the case of Theorem \ref{main2} the resolution of the
observational data needed for the algorithm to work for
noisy measurements is roughly the same as without noise;
however, to remove the exponential dependency on the
Grashof number in the error bounds,
Corollary \ref{cor1main2},
requires increasing
the resolution by its square.  Once the resolution
needed to remove the exponential term is achieved,
no further benefits are obtained by increasing the resolution.
To benefit from additional resolution in the observational
measurements, we note that oversampling
an already  very high resolution observation, and then  by locally averaging the oversampled observation, can
produce a observation that still has sufficient resolution but
with reduced variance in the noise.
In our case, we assumed the random errors were independent; however,
this may not be the case in practical problems.
For example, Budd, Freitag and Nichols \cite{budd2010} obtain
great benefits by using adaptive filters based on
assumptions about the independence of the measurements errors
in real-world weather forecasting applications.
The effect oversampling has on reducing the errors in our
theoretical bounds is consistent with the observed effects
of filtering in applications.

Computer simulations done by Gesho \cite{gesho2013} have shown
that in the absence of measurements errors the algorithm studied
in this paper performs much better than analytical estimates would
suggest.  In the case of nodal measurements, the actual
resolution requirements for the observation density is orders
of magnitude less than the upper bounds given by the analysis.
This phenomenon, that the numerics perform much better than the
analysis, was also noted for a different data
assimilation algorithm in~\cite{olson2003} and~\cite{olson2008}.
It is plausible that in the presence of stochastic noise
the data algorithm studied here will also perform numerically
much better than our analytic bounds.
Work is underway to study the numerical performance of this
data assimilation algorithm when the observation density is
much less than our analytic bounds and to understand how the
variance in the stochastic noise numerically affects the
convergence of the approximating solution to the reference
solution over time.

\section*{Acknowledgements}
The authors would like to dedicate this work to  Professor Ciprian Foias on the occasion of his 80th birthday as a token of appreciation for his continuous support, friendship  and inspiration.
H.B.~and E.S.T.~are thankful to the IMA for its kind hospitality where this project has started during their joint visit to the IMA.
E.S.T.~is also thankful to the warm hospitality of the  Instituto
Nacional de Mate\-m\'{a}tica  Pura  e Aplicada (IMPA), where part of
this work was completed.  H.B.~was supported in part by the Simons Foundation
grant \#283308. The work of E.O.~was supported in part by
EPSRC grant EP/G007470/1 and by sabbatical leave from the University of
Nevada Reno. The work of E.S.T.\ was supported in part by  the NSF grants DMS-1009950, DMS-1109640 and DMS-1109645. Also by the CNPq-CsF grant \#401615/2012-0, through the program Ci\^encia sem Fronteiras.

\end{document}